\documentclass{amsart}

\usepackage{amssymb}
\usepackage{amsthm}
\usepackage{amsmath} 
\usepackage{amsbsy}
\usepackage{bm}
\usepackage{hyperref}
\usepackage{tikz}
\newtheorem{theorem}{Theorem}

\newtheorem{corollary}{Corollary}
\newtheorem{lem}{Lemma}
\newtheorem*{lem*}{Lemma}

\usepackage{hyperref} 
\hypersetup{
    colorlinks=true,       
    linkcolor=blue,          
    citecolor=magenta,        
    filecolor=magenta,      
    urlcolor=cyan           
}

\numberwithin{equation}{section}

\def\f{\frac}

\def\({\left(}
\def \){ \right)}

\def\Bl{\Bigl}
\def\Br{\Bigr}
 
 \def\ee{{\textnormal{e}}}
 \def\i{{\textnormal{i}}}
 
 \def\Ga{\Gamma}

\def\da{{\delta}}

 \def\t{{\theta}}

 \def\s{{\sigma}}

 \def\NN{{\mathbb N}}

 \def\RR{{\mathbb R}}
 \def\SS{{\mathbb S}}

  \def\sph{\mathbb{S}^{d}}

\newcommand{\wh}{\widehat}

\def\ld{\lambda}

\def\be{\begin{equation}}
\def\ee{\end{equation}}

\title{General and refined Montgomery Lemmata}

\author{Dmitriy Bilyk}
\address{School of Mathematics, University of Minnesota, Minneapolis, MN 55408, USA.}
\email{dbilyk@math.umn.edu}

\author{Feng Dai}
\address{Department of Mathematical and Statistical Sciences\\
University of Alberta\\ Edmonton, Alberta T6G 2G1, Canada.}
\email{fdai@ualberta.ca}

\author{Stefan Steinerberger}
\address{Department of Mathematics\\ Yale University\\ New Haven, CT 06510, USA}
\email{stefan.steinerberger@yale.edu}

\subjclass[2010]{11L99, 11K38, 35B05, 35B40, 42B05, 52C35}
\keywords{Riesz energy, Discrepancy, Exponential sums, Stolarsky principle, Irregularities of Distribution, Spherical Cap Discrepancy, Beck gain.}

\begin{document}
\begin{abstract}
Montgomery's Lemma on the torus $\mathbb{T}^d$ states that a sum of $N$ Dirac masses cannot be orthogonal to many low-frequency trigonometric functions in a quantified way. We provide an extension to general manifolds that also allows for positive weights: let $(M,g)$ be a smooth compact $d-$dimensional manifold without boundary, let $(\phi_k)_{k=0}^{\infty}$ denote
the Laplacian eigenfunctions, let $\left\{ x_1, \dots, x_N\right\} \subset M$ be a set of points and $\left\{a_1, \dots, a_N\right\} \subset \mathbb{R}_{\geq 0}$ be a sequence of nonnegative weights. Then
 $$\sum_{k=0}^{X}{ \left| \sum_{n=1}^{N}{ a_n \phi_k(x_n)} \right|^2} \gtrsim_{(M,g)}    \left(\sum_{i=1}^{N}{a_i^2} \right) \frac{ X}{(\log{X})^{\frac{d}{2}}}.$$
This result is sharp up to the logarithmic factor.  Furthermore, we  prove a refined spherical version of  Montgomery's Lemma, and provide applications to estimates  of discrepancy and discrete energies of $N$ points on the sphere    $\mathbb{S}^{d}$.
\end{abstract}
\maketitle

\section{Introduction}
\subsection{Montgomery's Lemma.} The lemma, which constitutes the main subject of our investigation, has its origins in the theory of {\it{irregularities of distribution}}.  
Let $\left\{ x_1, \dots, x_N\right\} \subset \mathbb{T}^2 \cong [0,1)^2$ be a set of  $N$ points. Montgomery's theorem \cite{mont1} (see also Beck \cite{beck1, beck}) guarantees the existence of  a disk $D \subset \mathbb{T}^2$ with radius $1/4$ or $1/2$ such that the proportion of points in the disk is either much larger or much smaller than what  is predicted by the area
\begin{equation}\label{e.montdisc}
 \left| \frac1{N} \cdot \# \left\{1 \leq i \leq N: x_i \in D\right\} -  |D|\right| \gtrsim N^{-3/4}.
 \end{equation}
Higher-dimensional version of this statement  for sets in $\mathbb T^d$ holds with the right-hand side of the order $\displaystyle{N^{-\frac12 - \frac1{2d}}}$. 
The proof of Montgomery's argument proceeds as follows: we first bound the $L^{\infty}$-norm of the  `discrepancy function'  trivially from
below by the $L^2-$norm and then use Parseval's identity to multiplicatively separate the Fourier transform
of the characteristic function of the geometric shape (in the example above: a disk) and the Fourier coefficients of the Dirac measures located at $\left\{ x_1, \dots, x_N\right\} \subset \mathbb{T}^2$
$$ \widehat{     \left( \sum_{n=1}^{N}{\delta_{x_n}} \right)  } \,\,   (k) = 
\sum_{n=1}^{N}{ e^{-2 \pi i \left\langle k, x_n \right\rangle}} \,\,\, \textup{ for } \, k \in \mathbb{Z}^2 .$$
A fundamental ingredient of the method is the fact that the Fourier transform of finite set of Dirac measures cannot be too small on low frequencies.

\begin{lem*}[Montgomery \cite{mont1}] For any $\left\{ x_1, \dots, x_N\right\} \subset \mathbb{T}^2$ and $X \geq 0$
 \begin{equation}\label{e.mont} 
 \sum_{|k_1| \leq X} \sum_{|k_2| \leq X}{ \left| \sum_{n=1}^{N}{ e^{2 \pi i \left\langle k, x_n \right\rangle}}\right|^2} \geq N X^2.
 \end{equation}
\end{lem*}
This inequality is a two-dimensional analogue of an earlier result of Cassels \cite{cassels} and related to a result of Siegel \cite{siegel}.
Montgomery's Lemma is essentially sharp, generalizations  of the statement to $\mathbb{T}^d$ are straightforward.
This discussion suggests that expression akin to the left-hand side of \eqref{e.mont} can be used as  measures of uniformity of discrete sets of  points, much like the discrepancy  \eqref{e.montdisc}, see \cite{LuSt}. 

\subsection{Related recent results}
A slight sharpening of Montgomery's Lemma has recently been given by the third author in \cite{stein1}  (we only
describe the result on $\mathbb{T}^2$,  but higher-dimensional versions also hold): 
for all $\left\{ x_1, \dots, x_N\right\} \subset \mathbb{T}^2$ and $X \geq 0$
 \begin{equation}\label{e.St}
 \sum_{\|k\| \leq X}{ \left| \sum_{n=1}^{N}{ e^{2 \pi i \left\langle k, x_n \right\rangle}}\right|^2} \gtrsim \sum_{i,j=1}^{N}{ \frac{X^2}{1 +  X^4 \|x_i -x_j\|^4}}.
 \end{equation}
This quantifies the natural notion that any type of clustering of the points is going to decrease the orthogonality to trigonometric functions. 
Montgomery's Lemma has usually been regarded as an inequality on the torus as opposed to a more general principle. However, in the study of irregularities of
distribution on the sphere $\mathbb{S}^{d}$, the natural analogue of Fourier series is given by harmonic polynomials which are also well understood and allow
for fairly explicit analysis. In \cite{bil2} the first and second author proved a generalization of \eqref{e.montdisc} on $\mathbb S^d$, which essentially boiled down to a spherical analogue of \eqref{e.mont}. Namely, denoting the eigenfunctions of the spherical  Laplacian (i.e. spherical harmonics) by $\phi_0, \dots, \phi_k, \dots$, this inequality states 
\begin{equation}\label{e.bd} 
\sum_{k=0}^{X}{ \left| \sum_{n=1}^{N}{ \phi_k(x_n)} \right|^2} \gtrsim_d N X
\end{equation}

We observe that $\phi_0$ is constant and thus the first term is already of size $\sim N^2$. Exactly like on $\mathbb{T}^d$, for $k=0$  the inner sum is of size $N^2$ and the inequality is only interesting when
the number of eigenfunction $X$ starts to outnumber the number of points $X \gtrsim N$. This is also necessary because there are point sets that are orthogonal to the first
$\sim N$ eigenfunctions (this is classical on $\mathbb{T}^d$ and a substantial result on $\mathbb{S}^{d}$, see \cite{ahrens, bond}; it is likely to hold at a much greater
level of generality).

\section{Main results}\label{s.main}
In the present paper we further extend  Montgomery's Lemma in two different directions. First, we extend and generalize the statement of Montogomery's Lemma \eqref{e.mont} to general manifolds (with a  logarithmic loss). Second, in the case of the sphere $\mathbb S^d$, we combine the ideas of \eqref{e.St}-\eqref{e.bd} and prove a spherical analogue of \eqref{e.St}, which refines \eqref{e.bd}. We also provide several applications of this result to irregularities of distribution and energy minimization on the sphere: a notably example is a refinement of Beck's lower bound on the $L^2-$spherical cap discrepancy. 

\subsection{Montgomery Lemma on general manifolds.} We now phrase a general version of Montgomery's Lemma on compact manifolds. It relates
to various natural questions and we believe that a sharper form would be quite desirable. 

\begin{theorem}\label{t.1} Let $(M,g)$ be a smooth compact $d-$dimensional manifold, let $(\phi_k)_{k=0}^{\infty}$ denote
the $L^2-$normalized Laplacian eigenfunctions of $-\Delta_g$ with the corresponding eigenvalues arranged in increasing order.  Let $\left\{ x_1, \dots, x_N\right\} \subset M$, and let  $(a_i)_{i=1}^{N}$ be a set
of nonnegative weights. Then
 $$\sum_{k=0}^{X}{ \left| \sum_{n=1}^{N}{ a_n \phi_k(x_n)} \right|^2} \gtrsim_{(M,g)}  \left(\sum_{i=1}^{N}{a_i^2} \right) \frac{ X}{(\log{X})^\frac{d}{2}}.$$
\end{theorem}
It seems likely that the logarithm is an artifact of the method; the result is more general (but logarithmically worse) than the classical Montgomery Lemma on $\mathbb{T}^d$ and the version on the sphere \cite{bil2}
since it allows for nonnegative weights: the classical
proofs of Montgomery's Lemma, both on $\mathbb{T}^d$ and $\mathbb{S}^{d}$, fails in this more general setting. The last author has shown \cite{stein2} that, for $N$ sufficiently large, one of the summands for $X \leq c_d N$ is nonzero (where $c_d$ does not depend on the manifold).\\
Theorem 1 has various implications: one would naturally assume that as soon
as $X \gtrsim N$, the eigenfunctions should be fairly decoupled from the set of points and each single summand should be roughly of order
$\sim N$: the theorem shows this basic intuition to be true up to logarithmic factors. Another application concerns the limits of numerical integration:
the Laplacian eigenfunctions $\phi_k$ have mean value 0 as soon as $k \geq 1$ and are oscillating rather slowly. One would, of course, expect
it to be possible for $N$ points to integrate $\sim N$ functions exactly but, simultaneously, one would not expect such a rule to be able to do well
on a larger set of (mutually orthogonal) functions. This was shown to hold in \cite{stein2}, the formulation of Theorem~1 would lead to a more quantitative
result (akin to an estimate on the size of the unavoidable error, see also \cite{LuSt}).

\subsection{Spherical extensions of Montgomery's  Lemma}

We now restrict our attention to the case when  $M = \SS^d$ is  the unit sphere in $\RR^{d+1}$ equipped with the normalized Haar measure $\s$. Denote by $\mathcal{H}_n$ the space of all spherical harmonics of degree $n$ on $\SS^d$,
and let $\{Y_{n,k}:~ k=1, 2,\cdots, d_n\}$ be a real orthonormal basis of $\mathcal{H}_n$ (recall $\mbox{dim}~ \mathcal{H}_n \sim n^{d-1}$).  We have the following spherical analogue of \eqref{e.St}.

\begin{theorem}\label{thm-1-1} For $\{x_1, \cdots, x_N\}\subset \SS^d$, we have for all $L \in \mathbb{N}$
\begin{equation}\label{1-1}
   \sum_{n=0}^L \sum_{k=1}^{d_n} \Bl|\sum_{j=1}^N Y_{n,k}(x_j)\Br|^2 \ge c_d L^d  \sum_{i,j=1}^N \f { \log ( 2 + L \|x_i-x_j\|) }{(1+ L\|x_i-x_j\|)^{d+1}}.
\end{equation}

\end{theorem}

We observe that the left-hand side runs over $\sim L^d$ terms. Leaving just the diagonal terms ($i=j$) on the right-hand side one finds that the right-hand side is at least of the order   $\sim N L^d$, i.e. \eqref{1-1} is stronger than \eqref{e.bd}.  
Similar to the case of the torus, this result has immediate applications irregularities of distribution on the sphere. We provide refinements of both classical \cite{Beck2} and recent \cite{bil2} 
discrepancy bounds. Moreover,  with the help of the Stolarsky principle and its generalizations \cite{stol,bil2}, see \eqref{e.stol}-\eqref{e.gstol}, we obtain estimates on the the difference between discrete energies and energy integrals. These corollaries are gathered and proved in \S \ref{s.cor}.

\subsection{$L^2-$spherical cap discrepancy.}
We wish to highlight a particular implication that refines of a famous result of J. Beck \cite{Beck2}. The $L^2-$spherical cap discrepancy is defined as the $L^2-$norm of the spherical cap discrepancy (i.e. the difference between the empirical distribution of $N$ points and the uniform distribution) integrated over all radii (we refer to \S \ref{s.cor} for a more formal definition). The result of Beck states that for any set $Z$  of $N$ points on $\mathbb{S}^{d}$
$$
D_{L^2, \textup{cap}} (Z)  \gtrsim_d N^{-\frac12-\frac{1}{2d}}
$$
and this is sharp up to a logarithmic factor. Our approach yields a slight refinement.
\begin{theorem}\label{t.beck+} For any set of $N$ points $Z=\left\{z_1, \dots, z_N\right\} \subset \mathbb{S}^d$
$$D_{L^2, \textup{cap}} (Z) \gtrsim_d N^{-\frac12-\frac1{2d}} \left( \frac{1}{N} \sum_{i,j=1}^N \f {\log \,( 2 + N^{1/d} \|z_i-z_j\|)}{(1+ N^{1/d}\|z_i-z_j\|)^{d+1}}\right)^{1/2}.$$
\end{theorem}
We remark that summing over the diagonal $i=j$ shows that the additional factor is $\gtrsim 1$ implying Beck's original result. However, as soon as there is subtle clustering of points, the off-diagonal terms may actually contribute a nontrivial quantity.


\section{Montgomery Lemma on general manifolds: proof of Theorem \ref{t.1}.}
\begin{proof} We first observe that the eigenfunction $\phi_0 \equiv 1/\sqrt{|M|}$ is constant and thus
$$  \sum_{k=0}^{X} \left| \sum_{i=1}^{N}{ a_i \phi_k(x_i)} \right|^2 \gtrsim_{(M,g)}  \left(\sum_{i=1}^{N}{ a_i}\right)^2 = \frac{ \left(\sum_{i=1}^{N}{ a_i}\right)^2}{ \sum_{i=1}^{N}{ a_i^2}} \sum_{i=1}^{N}{ a_i^2}$$
and it thus suffices to prove the statement for
$$ X \gtrsim  \frac{ \left(\sum_{i=1}^{N}{ a_i}\right)^2}{ \sum_{i=1}^{N}{ a_i^2}}.$$
The proof starts by bounding the desired quantity from below; here, we let $t>0$ be an arbitrary number that will be fixed later.
\begin{align*}
 \sum_{k=0}^{X} \left| \sum_{i=1}^{N}{ a_i \phi_k(x_i)} \right|^2  &\geq  \sum_{k=0}^{X} e^{-\lambda_k t} \left| \sum_{i=1}^{N}{a_i \phi_k(x_i)} \right|^2  \\
&=   \sum_{k=0}^{X} e^{-\lambda_k t}  \sum_{i,j=1}^{N}{ a_i \phi_k(x_i) a_j \phi_k(x_j) } \\
&=  \sum_{i,j=1}^{N}{ a_i a_j \sum_{k=0}^{X} e^{-\lambda_k t} \phi_k(x_i) \phi_k(x_j)}.
\end{align*}
Here and throughout the proof, the $\ld_k$  denote the eigenvalues of $-\Delta_g$ such that $-\Delta_g \phi_k =\ld_k \phi_k$
and $0=\ld_0\leq \ld_1\leq \ld_2\leq \cdots$.
The inner sum is now close to a classical expansion for the heat kernel
$$ p_t(x,y) = \sum_{k=0}^{\infty} e^{-\lambda_k t} \phi_k(x) \phi_k(y).$$
This means that we can replace the inner sum by the heat kernel while incurring an error that only depends on the size of $X$. We will now make this precise: the main
ingredients are Weyl's law $ \lambda_k \sim c_{M} k^{2/d},$
where $c_{M}$ only depends on the volume of the manifold $M$ and H\"ormander's estimate \cite{hor}
$$ \|\phi_k\|_{L^{\infty}} \lesssim_{(M,g)} \lambda_k^{\frac{d-1}{4}}.$$
Combining these two inequalities, we can now estimate the tail:
\begin{align*}
 \left| \sum_{k=X+1}^{\infty} e^{-\lambda_k t} \phi_k(x_i) \phi_k(x_j) \right| &\lesssim_{(M,g)}  \sum_{k=X+1}^{\infty} \left|e^{- c k^{\frac{2}{d}} t} \phi_k(x_i) \phi_k(x_j) \right| \\
&\leq   \sum_{k=X+1}^{\infty} e^{- c k^{\frac{2}{d}} t} \|\phi_k\|_{L^{\infty}}^2\\
&\lesssim_{(M,g)} \sum_{k=X+1}^{\infty} e^{- ck^{\frac{2}{d}} t} \lambda_k^{\frac{d-1}{2}} \\
&\lesssim_{(M,g)} \sum_{k=X+1}^{\infty} e^{-c k^{\frac{2}{d}} t} k^{1 - \frac{1}{d}} .
\end{align*}
This quantity can be bounded from above by an integral which, after substitution, reduces to the incomplete Gamma function:
\begin{align*}
   \sum_{k=X+1}^{\infty} e^{- ck^{2/d} t}    k^{1 - \frac{1}{d}} &\leq  \int_{X}^{\infty} e^{- \left( \frac{y}{t^{-d/2}}\right)^{\frac{2}{d}} } y^{1 - \frac{1}{d}} dy \\
&=\frac{1}{t^{d-\frac{1}{2}}} \int_{cX t^{d/2}}^{\infty} e^{- z^{\frac{2}{d}} } z^{1 - \frac{1}{d}} dz \\
&=\frac d2 \frac{1}{t^{d-\frac{1}{2}}}  \Gamma\left(d-\frac{1}{2}, cX^{\frac{2}{d}} t\right).
\end{align*}
We will end up working in the regime $X^{\frac{2}{d}} t \gg 1$. In this regime, there is a classical asymptotic (see e.g. Abramowitz \& Stegun \cite[\S 6.5]{abra}), valid for $a \gg 1$,
$$  \Gamma\left(d-\frac{1}{2}, a \right) \lesssim_{d} a^{d-\frac{3}{2}} e^{-a}.$$
Altogether, this implies, since we may assume that
$$  X \gtrsim  \frac{ \left(\sum_{i=1}^{N}{ a_i}\right)^2}{ \sum_{i=1}^{N}{ a_i^2}},$$
the bound
\begin{align*}
  \sum_{k=0}^{X} \left| \sum_{i=1}^{N}{ a_i \phi_k(x_i)} \right|^2  &\gtrsim \sum_{i,j=1}^{N}{a_i a_j p_t(x_i, x_j) } -  C\sum_{i,j =1}^{N}{ \frac{ a_i a_j }{t^{d-\frac{1}{2}}} \left( X^{\frac{2}{d}} t\right)^{d - \frac{3}{2}} \exp\left(-c X^{\frac{2}{d}} t\right)} \\
&= \left( \sum_{i,j=1}^{N}{a_i a_j p_t(x_i, x_j) } \right) -  C\frac{ \left( \sum_{i=1}^{N}{a_i}\right)^2 }{t^{d-\frac{1}{2}}} \left( X^{\frac{2}{d}} t\right)^{d - \frac{3}{2}} \exp\left(- cX^{\frac{2}{d}} t\right) \\
&\gtrsim  \sum_{i,j=1}^{N}{a_i a_j p_t(x_i, x_j) } -  C\frac{ X \sum_{i=1}^{N}{a_i^2} }{t^{d-\frac{1}{2}}} \left( X^{\frac{2}{d}} t\right)^{d - \frac{3}{2}} \exp\left(-c X^{\frac{2}{d}} t\right)
\end{align*}
We will end up working at time $t \sim X^{-\frac{2}{d}} \log{X} \ll 1$ which, for $X$ sufficiently large, enables us to make use of Varadhan's short-time asymptotics
$$ p_t(x,y) \sim \frac{1}{(4 \pi t)^{d/2}} \exp\left( - \frac{\|x-y\|^2}{4t} \right)$$
to argue that
$$ \sum_{i,j=1}^{N}{a_i a_j p_t(x_i, x_j) }  \geq \sum_{i=1}^{N}{a_i^2 p_t(x_i, x_i) } \gtrsim t^{-\frac{d}{2}} \sum_{i=1}^{N}{a_i^2}.$$
Summarizing, we have
$$   \sum_{k=0}^{X} \left| \sum_{i=1}^{N}{ a_i \phi_k(x_i)} \right|^2 \gtrsim_{(M,g)} \sum_{i=1}^{N}{a_i^2} \left[ t^{-\frac{d}{2}} - \frac{C X  }{t^{d-\frac{1}{2}}} \left( X^{\frac{2}{d}} t\right)^{d - \frac{3}{2}} \exp\left(-c X^{\frac{2}{d}} t \right)  \right].$$
Setting $t =A X^{-\frac{2}{d}} \log{X}$ with $A=\f 1c (1-\f 1d) +1$ now implies the result.
\end{proof}


\section{An improved Montgomery Lemma on the sphere:\\ proof of Theorem \ref{thm-1-1}}

Let $C_n^\lambda$ denote the Gegenbauer (ultraspherical) polynomials of degree $n$, which are orthogonal on $[-1,1]$ with respect to the weight $w_\lambda(t) = (1-t^2)^{\lambda - 1/2} $ (see \cite{DX} for the backgound information). Since we are working on $\mathbb S^d$, we set  $\ld=\f {d-1}2$. Denote also $E_n^\ld(t)=\f {n+\ld}{\ld} C_n^\ld (t)$.  For $\da>0$, we define the Ces\`aro-type kernel 
$$ K_L^\da (t):=\sum_{k=0}^L \f {A_{L-k}^\da}{A_L^\da} E_k^\ld (t),\   \   \  \text{with}\   \  A_j^\da=\f {\Ga(j+\da+1)}{\Ga(j+1)\Ga(\da+1)}.$$

It is a classical result of Kogbetliantz \cite{kog} (see also \cite{reim}) that $K_L^\delta (t) \ge 0$ on $[-1,1]$, whenever $\delta \ge d$. 

\begin{lem}\label{lem-1-1} For $\{x_1, \cdots, x_N\}\subset \SS^d$ and any $\da>0$, we have
$$ \sum_{k=1}^{d_n} |\sum_{j=1}^N Y_{n,k}(x_j)|^2 =\sum_{i,j=1}^N E_n^\ld (x_i\cdot x_j)\ge 0,\   \    \  n=0,1,\cdots, $$
and
\begin{equation}\label{1-2}
   \sum_{n=0}^L \sum_{k=1}^{d_n} |\sum_{j=1}^N Y_{n,k}(x_j)|^2 \ge \sum_{i,j=1}^N K_{L}^\da (x_i\cdot x_j).
\end{equation}
\end{lem}

This lemma follows directly from  the addition formula for spherical harmonics. We include the proof here for the sake of completeness.

\begin{proof}By the addition formula for spherical harmonics, we have \begin{align*}
                   \sum_{k=1}^{d_n} |\sum_{j=1}^N Y_{n,k}(x_j)|^2&=  \sum_{k=1}^{d_n} \sum_{i=1}^N \sum_{j=1}^N Y_{n,k}(x_i)Y_{n,k}(x_j)=\sum_{i,j=1}^N \sum_{k=1}^{d_n} Y_{n,k}(x_i)Y_{n,k}(x_j)\\
                   &=\sum_{i,j=1}^N E_n^\ld (x_i\cdot x_j).
                   \end{align*}
                   This also implies that
                   \begin{align*}
                     \sum_{n=0}^L \sum_{k=1}^{d_n} |\sum_{j=1}^N Y_{n,k}(x_j)|^2&= \sum_{n=0}^L \sum_{i,j=1}^N E_n^\ld(x_i\cdot x_j)
                      \ge \sum_{n=0}^L \f {A_{L-n}^\da}{A_{L}^\da} \sum_{i,j=1}^N E_n^\ld(x_i\cdot x_j)\\
                      &= \sum_{i,j=1}^N \sum_{n=0}^L \f {A_{L-n}^\da}{A_{L}^\da} E_n^\ld(x_i\cdot x_j)
                      =\sum_{i,j=1}^N K_{L}^\da (x_i\cdot x_j).
                   \end{align*}
\end{proof}

Numerical experiments suggest that $K_n^d$ is not just non-negative, but is actually strictly positive and should satisfy favorable lower bounds. However, we could not prove it, hence, as in \cite{stein1}, we shall make use of additional rounds of averaging.
Define
$$ G_n^{d+1}(t)=\f 1{n+1} \sum_{j=0}^n K_j^d(t) \,\,\, \textup{ and }\,\,\, G_n^{d+2}(t)=\f 1{n+1} \sum_{j=0}^n G_j^{d+1}(t).$$

\begin{lem}\label{lem-1-2}For $n\in \NN$ and $\t\in (0, \pi)$,
\begin{equation}\label{1-3-0}
    G_n^{d+2} (\cos\t) \ge C n^d (1+n\t)^{-d-1}\log (2+n\t).
\end{equation}

\end{lem}

{\emph{Remark:}}  It seems that \eqref{1-3-0} with $G_n^{d+1}$ in place of $G_n^{d+2}$ remains true, but the proof would be more involved (we prove a slightly weaker bound \eqref{1-5}).  

\begin{proof}
  First, we recall that
  $ K_n^d (\cos\t)\ge 0$ for  $\t\in [0,\pi]$, and
      $\|K_n^d\|_\infty =K_n^d (1) \sim (n+1)^d$.
      It follows that for $\da=d+1$ or $ d+2$,
     $$\|G_n^\da \|_\infty =G_n^\da (1) \sim (n+1)^d.$$
      By Bernstein's inequality for trigonometric polynomials,  this also implies that for $F_n(t):=K_n^d(t)$ or  $ G_n^{d+1}(t)$ or  $ G_n^{d+2}(t)$, we have
     \begin{equation}\label{1-4}
         F_n (\cos \t) \ge \f 12 \|F_n\|_\infty \sim  (n+1)^d,\   \   0\leq \t\leq \f 1{2n}.
     \end{equation}

Next, we show that
\begin{equation}\label{1-5}
    G_n^{d+1} (\cos\t) \ge c n^d(1+n\t)^{-d-1},\  \  n\ge 1,\   \ \t\in [0,\pi].
\end{equation}
If $0\leq \t\leq \f 1{2n}$, then \eqref{1-5} follows directly from \eqref{1-4}.  For $\f 1{2n}\leq \t\leq \pi$, we have
  \begin{align*}
    G_n^{d+1} (\cos\t)& =\f 1{n+1} \sum_{j=0}^n K_j^d (\cos\t) \ge \f 1{n+1} \sum_{0\leq j \leq \f 1{2\t}} K_j^d (\cos\t)\\
    &\ge c \f 1{n+1} \sum_{0\leq j \leq \f 1{2\t}} (j+1)^d\sim n^{-1} \t^{-d-1} \sim n^d (1+n\t)^{-d-1}.
  \end{align*}

  Finally, we prove  estimate \eqref{1-3-0}. Note that \eqref{1-5} with $G_n^{d+2}$ in place of $G_n^{d+1}$ remains true. Thus,
  without loss of generality, we may assume that $\f {2}{n}\leq \t\leq \pi$ and $n\ge 10$. We then  have
 \begin{align*}
    G_n^{d+2} (\cos\t)& =\f 1{n+1} \sum_{j=0}^n G_j^{d+1} (\cos\t)\ge c n^{-1}\sum_{j=0}^n j^d(1+j\t)^{-d-1}\\
    &\ge c n^{-1}\sum_{\t^{-1}\leq j \leq n} j^{-1} \t^{-d-1}\ge c n^{-1}\t^{-d-1} \int_{\t^{-1}+1}^n \f {dt}t\\
    &=c n^{-1}\t^{-d-1} \int_{1+\t}^{n\t} \f {dt}t\sim n^d (1+n\t)^{-d-1} \log (n\t+2).
 \end{align*}
\end{proof}

\begin{proof}[Proof of Theorem \ref{thm-1-1}]
Using Lemma \ref{lem-1-1}, we have
\begin{align}
  \sum_{n=0}^L \sum_{k=1}^{d_n} |\sum_{j=1}^N Y_{n,k}(x_j)|^2 &
 \ge \f 1L \sum_{m=0}^L \sum_{n=0}^m \sum_{k=1}^{d_n} |\sum_{j=1}^N Y_{n,k}(x_j)|^2\notag\\
 &\ge \f 1L \sum_{m=0}^L \sum_{i,j=1}^N K_{m}^d (x_i\cdot x_j)=\sum_{i,j=1}^N G_{L}^{d+1} (x_i\cdot x_j).\label{1-6}
\end{align}
Using \eqref{1-6} and   averaging once again, we have
\begin{align*}
  \sum_{n=0}^L &\sum_{k=1}^{d_n} |\sum_{j=1}^N Y_{n,k}(x_j)|^2
 \ge \f 1L \sum_{m=0}^L \sum_{n=0}^m \sum_{k=1}^{d_n} |\sum_{j=1}^N Y_{n,k}(x_j)|^2\\
 &\ge \f 1L \sum_{m=0}^L \sum_{i,j=1}^N G_{m}^{d+1} (x_i\cdot x_j)=\sum_{i,j=1}^N G_{L}^{d+2} (x_i\cdot x_j),
\end{align*}
which, using \eqref{1-3-0}, implies the desired estimate \eqref{1-1}.

\end{proof}

\section{Some  corollaries for discrepancy and discrete energy of point distributions on the sphere}\label{s.cor} 

For  a finite set of points  $Z=\{z_1, \cdots, z_N\}\subset \sph$,  its $L^2-$discrepancy with respect to a function  $f: [-1,1] \rightarrow \mathbb R$  is defined as 
\begin{align}\label{e.dL2}
    D_{L^2, f}(Z) =\bigg( \int\limits_{\sph}\Bl| \f 1N \sum_{j=1}^N f(x\cdot z_j) -  \int\limits_{\sph}  f(x\cdot y)\, d\s(y) \Br|^2\, d\s(x)\bigg)^{\f12}.
\end{align}
In particular, when $f(t)= f_\tau (t)  = {\bf{1}}_{[\tau,1]} (t)$, one obtains the discrepancy with respect to spherical caps $C(x,\tau) = \{ y \in \sph:\, x\cdot y \ge \tau \}$ of aperture $\arccos \tau$, i.e. 
\begin{equation}
 D_{L^2, f_\tau}^2 (Z) =  \int\limits_{\sph}\Bl| \f 1N \sum_{j=1}^N {\bf{1}}_{C(x,\tau)} (  z_j) -  \sigma \big( C(x,\tau) \big)  \Br|^2\, d\s(x),
\end{equation}
Its  $L^2-$average    over the parameter $\tau$ yields the  classical $L^2-${\it{spherical cap discrepancy}} 
\begin{align}\label{e.disccap}
D_{L^2, \textup{cap}}^2 (Z) =     \int\limits_{-1}^1  D_{L^2, f_\tau}^2 (Z)  \, d\tau  , 
\end{align}
which has been extensively studied \cite{beck1,Beck2}. In particular, this quantity satisfies the following identity known as the {\it{Stolarsky principle}} \cite{stol}, which relates  it to a certain discrete energy.
\begin{equation}\label{e.stol}
c_d \, D_{L^2, \textup{cap}}^2 (Z)  = \int\limits_{\mathbb S^{d}} \int\limits_{\mathbb S^{d}} \|  x- y \| \, d\sigma (x)\, d\sigma (y)\,\,  - \,\, \frac{1}{N^2} \sum_{i,j = 1}^N \| z_i - z_j \| ,
\end{equation}
where $c_d$ is a dimensional constant.  It has been established in \cite{bil2,bil3} that  Stolarsky principle can be generalized in the following way: for $f \in L^2 \big([-1,1],  w_\lambda \big)$ 
\begin{align}\label{e.gstol}
    D^2_{L^2, f}(Z) =  \frac{1}{N^2} \sum_{i=1}^N \sum_{j=1}^N F(z_i\cdot z_j) - \int\limits_{\sph}\int\limits_{\sph} F(x\cdot y)\, d\s(x) d\s(y),
\end{align}
where the function $F: [-1,1]\rightarrow \mathbb R$ is defined through the identity 
\begin{equation}\label{e.Ff}
\widehat{F} (n,\lambda) = \big( \widehat{f} (n,\lambda) \big)^2 .
\end{equation}  
Here and throughout the proof, 
$$ \wh{f}(n,\ld):=\f {(n+\ld)\Ga(\ld)}{\sqrt{\pi}\Ga(\ld+\f12) }  \int_{-1}^1 f(t) C_n^\ld(t) (1-t^2)^{\ld-\f12}\, dt.$$
It is now easy to see that the refined spherical Montgomery Lemma, Theorem \ref{thm-1-1},   provides new estimates both for the discrepancy and discrete energies. Setting 
$$\displaystyle{G(x) = \frac1{N} \sum_{j=1}^N f (x\cdot z_j)},~ \mbox{we see that} \quad
D_{L^2, f}(Z)  = \| G - \widehat{G} (0,\ld) \|_{L^2 (\sph, d\sigma)}$$ and, according to the Funk--Hecke formula, for any spherical harmonic $Y_n \in \mathcal H_n$
\begin{equation}
\langle G, Y_n \rangle = \frac1{N} \sum_{j=1}^N \int\limits_{\sph}  f(x\cdot z_j ) Y_n (x) d\sigma (x) = \frac{1}{N} \widehat{f} (n,\lambda) \sum_{j=1}^N Y_n(z_j).
\end{equation}
Thus we find that 
\begin{align}
D_{L^2, f}^2(Z) &  = \| G - \widehat{G} (0,\ld) \|_2^2  =  \sum_{n=1}^\infty \sum_{k=1}^{d_n}  |\langle G, Y_{n,k} \rangle|^2 \\
\nonumber & = \frac1{N^2} \sum_{n=1}^\infty \big|  \widehat{f} (n,\lambda) \big|^2  \sum_{k=1}^{d_n}  \bigg| \sum_{j=1}^N  Y_{n,k} (z_j)   \bigg|^2 \\ 
\nonumber & \ge \frac{1}{N^2 } \cdot \min_{1\le n \le L}  \big|  \widehat{f} (n,\lambda) \big|^2  \cdot \sum_{n=1}^L \sum_{k=1}^{d_n}  \bigg| \sum_{j=1}^N  Y_{n,k} (z_j)   \bigg|^2\\
\nonumber &  =  \frac{1}{N^2 } \cdot \min_{1\le n \le L}  \big|  \widehat{f} (n,\lambda) \big|^2  \cdot  \left( \sum_{n=0}^L \sum_{k=1}^{d_n}  \bigg| \sum_{j=1}^N  Y_{n,k} (z_j)   \bigg|^2  - N^2 \right),
\end{align}
where we used the fact that the term, corresponding to $n=0$, is $N^2$.  If we set $L= C' N^{\frac1{d}}$ with $C'$ being a large dimensional constant,  and leave just the diagonal terms in \eqref{1-1}, we see that $$ \sum_{n=1}^L \sum_{k=1}^{d_n}  \bigg| \sum_{j=1}^N  Y_{n,k} (z_j)   \bigg|^2  \ge c'' N^2.$$
Therefore, again applying \eqref{1-1} of Theorem \ref{thm-1-1}, we arrive at the following corollary:

\begin{corollary}\label{c.1}
Let   $f \in L^2 \big([-1,1],  (1-t^2)^{\lambda - \frac12} \big)$. For $Z=\{z_1,\dots, z_N \} \subset \sph$ we have 
\begin{equation}\label{e.c1}
D_{L^2, f}^2 (Z) \gtrsim \frac{1}{N } \cdot \min_{1\le n \le C' N^{\frac1{d}} }  \big|  \widehat{f} (n,\lambda) \big|^2  \cdot  \sum_{i,j=1}^N \f {\log ( 2 + N^{1/d} \|z_i-z_j\|)}{(1+ N^{1/d}\|z_i-z_j\|)^{d+1}},
\end{equation}
where  $C'$ is a large constant depending only on the dimension. 
\end{corollary}

Such lower bounds, which show that finite point sets cannot be distributed too uniformly, are a common theme in the subject of {\emph{irregularities of distribution}}.  Using the generalized Stolarsky principle \eqref{e.gstol} and relation \eqref{e.Ff} we can also obtain  a similar corollary for the discrete energy:
\begin{corollary}\label{c.2}
Assume that $F\in C[-1,1]$ and $\widehat{F} (n,\lambda) \ge 0$ for all $n\ge 1$ (i.e., up to the constant term, $F$ is a positive definite function on the sphere $\mathbb S^d$). Then for any point distribution  $Z=\{z_1,\dots, z_N \} \subset \sph$ 
\begin{equation}\label{e.c2}
\frac{1}{N^2} \sum_{i,j=1}^N  F(z_i\cdot z_j) -  I_F (\sigma)   \gtrsim \frac{1}{N } \cdot \min_{1\le n \le C' N^{1/d} }     \widehat{F} (n,\lambda)   \cdot  \sum_{i,j=1}^N \f {\log ( 2 + N^{1/d} \|z_i-z_j\|)}{(1+ N^{1/d}\|z_i-z_j\|)^{d+1}},
\end{equation}
where  $C'$ is a large constant depending only on the dimension,  and $ I_F (\sigma) = \int\limits_{\sph}\int\limits_{\sph} F(x\cdot y)\, d\s(x) d\s(y) $ denotes the  energy integral with potential given by $F$. 
\end{corollary}

\noindent {\emph{Remark:}} The fact that every continuous positive definite function on the sphere can be represented by \eqref{e.Ff}, i.e. has appropriate decay of $\widehat{F} (n,\lambda)$, has been discussed in \cite[Lemma 2.3]{bil2}.\\

It is known (see e.g. \cite{bil2,bil3}) that for positive definite functions $F$, the uniform surface measure $\sigma$ minimizes the energy with potential $F$ over all Borel probability measures on $\mathbb S^d$. Thus Corollary \ref{c.2} states, in a quantitative way, that the energy of finite atomic measures with equal weights cannot be too close to the minimum. 

We observe that leaving just the $N$ diagonal terms ($i=j$) in the right-hand sides of \eqref{e.c1} and \eqref{e.c2} we recover the bounds obtained in \cite[Theorem 4.2]{bil2}:
\begin{align}
D_{L^2, f} (Z)  &\gtrsim   \min_{1\le n \le C' N^{\frac1{d}} }  \big|  \widehat{f} (n,\lambda) \big|,  \\
\nonumber \frac{1}{N^2} \sum_{i,j=1}^N  F(z_i\cdot z_j) -  I_F (\sigma)  &   \gtrsim   \min_{1\le n \le C' N^{\frac1{d}}}     \widehat{F} (n,\lambda).
\end{align}
Corollaries \ref{c.1} and \ref{c.2} add more subtle information to these  lower bounds. \\

Returning to the classical case of the spherical cap discrepancy \eqref{e.disccap},  recall that Beck's  famous result \cite{beck}, which states that 
 \begin{equation}\label{e.beck}
D_{L^2, \textup{cap}} (Z)  \gtrsim N^{-\frac12-\frac{1}{2d}}
 \end{equation} 
for any $N$-point set in the sphere $\sph$ (and this  is optimal up to a logarithmic factor). Using the fact that  (see e.g. \cite{Sz} or \cite{bil2}) 
\begin{equation}
\int_{-1}^1 \big| \widehat{f_\tau} (n,\lambda) \big|^2  \, d\tau \approx n^{-d-1}
\end{equation}
 and  repeating the arguments above almost verbatim, but with an additional averaging in $\tau$, one obtains a refinement of Beck's original estimate (this refinement has been stated in \S \ref{s.main} as Theorem \ref{t.beck+}). 

\begin{corollary}\label{c.3}
For any point distribution  $Z=\{z_1,\dots, z_N \} \subset \sph$ 
\begin{equation}\label{e.c3}
D_{L^2, \textup{cap}}^2 (Z) \gtrsim_d N^{-2-\frac1{d}} \cdot \sum_{i,j=1}^N \f {\log \,( 2 + N^{1/d} \|z_i-z_j\|)}{(1+ N^{1/d}\|z_i-z_j\|)^{d+1}}.
\end{equation}
\end{corollary}
As before, by considering only the diagonal terms one recovers Beck's result \eqref{e.beck}, and the bound \eqref{e.c3} provides more information: in particular, if the order of magnitude of the energy on the right-hand side is significantly greater than $N$, then the spherical cap discrepancy of $Z$ is necessarily too big.
The original Stolarsky principle \eqref{e.stol} then leads to the following corollary concerning the sum of Euclidean distances between $N$ points on the sphere:
\begin{corollary}
For any point distribution  $Z=\{z_1,\dots, z_N \} \subset \sph$ 
\begin{equation}\label{e.c3}
\mathcal J_d - \,\, \frac{1}{N^2} \sum_{i,j = 1}^N \| z_i - z_j \| \gtrsim_d N^{-2 - \frac1{d} } \cdot \sum_{i,j=1}^N \f {\log \,  ( 2 + N^{\f 1d} \|z_i-z_j\|)}{(1+ N^{\f 1d}\|z_i-z_j\|)^{d+1}},
\end{equation}
where  
 $$ \mathcal J_d = \int\limits_{\mathbb S^{d}} \int\limits_{\mathbb S^{d}} \|  x- y \| \, d\sigma (x)\, d\sigma (y)\,\, =   \frac{ 2^d \big[ \Gamma \big(\frac{d+1}{2}\big) \big]^2 }{\sqrt{\pi} \Gamma \big( d+ \frac12 \big) }.$$
 \end{corollary}

\textbf{Acknowledgment.} Parts of this work were started at the Workshop ``Discrepancy Theory and Quasi-Monte Carlo methods" held at the Erwin Schr\"odinger Institute, September 25 -- 29, 2017. The authors
gratefully acknowledge its hospitality.  Bilyk's work is supported by NSF grant DMS 1665007.

\end{document}